\documentclass[11pt]{amsart}
\usepackage{latexsym}
\usepackage{amsfonts}

\usepackage{amsmath,amsthm,amssymb}

\newtheorem{thm}{Theorem}[section]
\newtheorem{lem}[thm]{Lemma}
\newtheorem{cor}[thm]{Corollary}

\newtheorem{prop}[thm]{Proposition}
\theoremstyle{definition}
\newtheorem{defn}[thm]{Definition}

\newtheorem{obs}[thm]{Observation}
\newtheorem{claim}[thm]{Claim}
\newtheorem{qu}[thm]{Question}

\newcommand{\comment}[1]{}

\begin{document} 

\author[Carl G. Jockusch, Jr.]{Carl G. Jockusch, Jr.}

\address{\tt Department of Mathematics, University of Illinois at
  Urbana-Champaign, 1409 West Green Street, Urbana, IL 61801, USA
  \newline http://www.math.uiuc.edu/\~{}jockusch/} \email{\tt
  jockusch@math.uiuc.edu}

\author[Paul Schupp]{Paul E. Schupp}

\address{\tt Department of Mathematics, University of Illinois at
  Urbana-Champaign, 1409 West Green Street, Urbana, IL 61801, USA}
  \email{\tt schupp@math.uiuc.edu}

\title[Asymptotic density and computability]{Asymptotic density and the theory of  computability : A partial survey}

\dedicatory{This paper is dedicated to Rod Downey\\ in honor of his important contributions to computability theory}

\maketitle


\section{Introduction}

   The purpose of this paper is to survey  recent work on how classical asymptotic density interacts with the theory of computability.    We have tried to make  the survey  accessible to those who are not specialists in computability theory and we mainly state results without proof, but we include a few easy proofs to illustrate the flavor of the subject.

  In complexity theory, classes such as $\mathcal{P}$ and $\mathcal{NP}$ are defined by using worst-case measures.
That is, a problem belongs to the class if there is an algorithm solving it which has a suitable bound on its running time over \emph{all} instances of the problem.   
Similarly,  in computability theory,  a problem is classified as computable if there is a single algorithm which solves all instances of the given problem. 

   There is now a general 
awareness that worst-case measures may not give a good picture of a particular algorithm or problem
since hard instances may be very sparse.
  The  paradigm case is  Dantzig's  Simplex Algorithm (see \cite{Cor})  for linear programming problems.  
This algorithm  runs many hundreds of times every day for scheduling and transportation
problems, almost always very quickly.
  There are clever examples of Klee and Minty \cite{KM} showing that there exist instances
for which the Simplex Algorithm must  take exponential time,  but such examples are not  encountered in practice.

   Observations of this type led to the
development of {\it  average-case complexity} by Gurevich \cite{Gurevich} and by Levin \cite{Levin} independently.
There are  different approaches to the average-case complexity, but they
all involve computing the expected value of the running time of an
algorithm with respect to some measure on the set of inputs.  Thus the problem must be decidable and one still needs to know 
the worst-case complexity.

  Another  example of hard instances being sparse is the behavior of algorithms for decision problems in group theory
used in computer algebra packages.   There is often  some kind of an easy
``fast check'' algorithm which quickly produces a solution for
``most'' inputs of the problem. This is true even if the
worst-case complexity of the  particular problem is very high or the
problem is even unsolvable. Thus many  group-theoretic
decision problems  have  a very large set of inputs where the
(usually negative) answer can be obtained easily and quickly. 
   
  Such examples led   Kapovich, Myasnikov, Schupp and  Shpilrain \cite{KMSS}  to introduce generic-case complexity
  as a complexity measure which is often more useful and easier to work with than either worst-case or average-case complexity.    In generic-case complexity, one considers algorithms which answer correctly within a given time bound  on a set of inputs of asymptotic density $1$.  They   showed that many 
classical decision problems in group theory resemble the situation of the Simplex Algorithm
in that hard instances are very rare.  For example, consider the word problem for one-relator groups.   In the 1930's Magnus (see \cite{LS})  showed that this problem is decidable but we still have no idea of the possible worst-case complexities over the whole
class of one-relator groups.  However, for \emph{every} one-relator group with at least three generators,
the word problem is generically linear time by Example 4.7 of \cite{KMSS}.  Also, in the famous groups of Novikov \cite{Novikov} and
Boone (see \cite{Rotman}) with undecidable word problem,  the word problem has linear time generic-case complexity by Example 4.6 of \cite{KMSS}.

 Although it focused on complexity, the paper \cite{KMSS}
introduced a general definition of generic computability in Section 9.

   Let $\Sigma$ be a nonempty finite alphabet  and let
$\Sigma^*$ denote the set of all finite words on  $\Sigma$.
The \emph{length}, $|w|$,  of a word  $w$  is the number of
letters in $w$.
Let $S$ be a subset of
$\Sigma^*$. For every $n \geq 0$ let $S \rceil n$ denote  the set of all  words in
$S$  of length  less than or equal to $n$.  In this situation we can copy the classical definition
of asymptotic density from number theory.

\begin{defn}For every $n \geq 0$, the \emph{density of $S$ up to  $n$} is
     \[ \rho_n(S) = \frac{|S  \rceil n|}{|\Sigma^*  \rceil n|} \]

The \emph{density} of $S$ is 

$$  \rho(S) = \lim_{n \to \infty} \rho_n(S)  $$
if this limit exists.
\end{defn}

\begin{defn}   Let $S \subseteq \Sigma^*$.    We say that $S$ is \emph{generic} if $\rho(S) = 1$
and $S$ is \emph{negligible} if $\rho(S) = 0$.
\end{defn}

It is clear that $S$ is generic if and only if its complement
$\overline{S} = \Sigma^* \setminus S$ is negligible. Also, the intersection (union) of finitely
many generic (negligible) sets is generic (negligible).    This notion of genericity should not be confused with notions of genericity from forcing in computability theory and set theory.   The latter are related to  Baire category rather than density.

\begin{defn}(\cite{KMSS}) Let $S$ be a subset of $\Sigma^*$ with characteristic
  function $\chi_S$.   A set $S$ is  \emph{generically   computable}
 if there exists a \emph{partial computable function}  $\varphi$
   such that  $\varphi (x)  = \chi_S (x)$ whenever $\varphi(x)$ is defined
  (written $\varphi(x) \downarrow$) and  the domain of $\varphi$ is
  generic in $\Sigma^*$. 
\end{defn}

 We stress that \emph{all}
  answers given by $\varphi$ must be correct even though $\varphi$ need not
  be everywhere defined, and, indeed, we do not require the domain of
  $\varphi$ to be computable.  In studying complexity we can clock the partial
algorithm and consider it as not answering if it does not answer within
the allotted amount of time.

   To illustrate that even undecidable problems may be generically easy, we consider 
the \emph{Post Correspondence Problem} (PCP).  Fix a
finite alphabet $\Sigma$ of size $k \geq 2$. A typical instance of the
problem consists of a finite sequence of pairs of words $(u_1, v_1),
(u_2, v_2), \dots, (u_n, v_n)$ , where $u_i, v_i \in \Sigma^*$ for $1
\leq i \leq n$.  The problem is to determine whether or not there is a finite
nonempty sequence of indices $i_1, i_2, \dots, i_k$ such that
$$u_{i_1} u_{i_2} \dots u_{i_k} = v_{i_1} v_{i_2} \dots v_{i_k}$$
holds.

In other words, can finitely many  $u$'s be concatenated to
give the same word as the corresponding concatenation of $v$'s? 
Emil Post proved in 1946 \cite{P} that this problem is unsolvable for each alphabet
$\Sigma$ of size at least $2$ and this result   has been used to show that many
other problems are unsolvable.  Our exposition of  a fast generic algorithm for
the PCP  follows  the book \cite{MSU}
  by  Myasnikov,  Shpilrain,
and  Ushakov.

The generic algorithm works as follows.
Say that two words $u$ and $v$ are \emph{comparable} if either is a prefix of the other.
Given an instance $(u_1, v_1), (u_2, v_2), \dots, (u_n, v_n)$ of the
PCP determine whether or not $u_i$ and $v_i$ are comparable for some
$i$ between $1$ and $n$.  If not, output ``no''.  Otherwise, give no
output.

If the given instance has a solution $u_{i_1} \dots u_{i_n} = v_{i_1}
\dots v_{i_n}$, then $u_{i_1}$ and $v_{i_1}$ must be comparable.
Hence the above algorithm never gives a wrong answer.

  We now show that the algorithm  gives an answer with density $1$ on the
natural stratification of instances of the problem.
 Let $I_s$ be the set of instances $(u_1, v_1),
(u_2, v_2), \dots, (u_n, v_n)$ where $n \leq s$ and each word $u_i, v_i$
 has length at most $s$.
Each $I_s$ is finite, each $I_j \subseteq I_{j+1}$  and every instance of the PCP
belongs to some $I_s$.   Let $D_s$ be the set of instances in $I_s$ for which
the algorithm gives an output.

\begin{claim} \quad  $\lim_s \frac{|D_s|}{|I_s|} = 1$
\end{claim}

\begin{proof} Put the uniform measure on $I_s$ and let an element $(u_1, v_1),
(u_2, v_2), \dots, (u_n, v_n)$ of $I_s$ be chosen uniformly at random.
To prove the claim, we show that the probability that the algorithm
diverges on a random element of $I_s$ approaches $0$ as $s$ approaches
infinity.

For any fixed values of $v_1, u_2, \dots, v_n$ the conditional probability
that $u_1$ is a prefix of $v_1$ is at most $\frac{s+1}{2^s}$ since there at
least  $2^s$ words on $\Sigma$ of length $s$ and at most $s+1$ of these are prefixes
of $v_1$.

Hence, the probability that $u_1$ is a prefix of $v_1$ is at most
$\frac{s+1}{2^s}$, and the probability that some $u_i$ is comparable with
$v_i$ is at most $\frac{2s(s+1)}{2^s}$.  So the probability that the
  algorithm gives no answer on the given instance is at most
  $\frac{2s(s+1)}{2^s}$, which tends to $0$ as $s$ approaches infinity.
\end{proof}

  The  generic algorithm we described works in quadratic time,
so the generic-case complexity of the  Post Correspondence Problem is at most quadratic time.

  From now on we mainly consider subsets of the  the set $\mathbb{N} = \{0,1,\dots\}$ of natural numbers,
which we identify with the set $\omega$ of finite ordinals,   In terms of the preceding definitions,  we are using the $1$-element
alphabet $\Sigma = \{1\}$ and identifying $n \in \omega$ with its
unary representation $1^n \in \{1\}^*$.  In this context, 
we are  using  classical asymptotic density.  
If $A \subseteq \mathbb{N}$, then, for $ n \ge 1$, the \emph{density of $A$ below $n$} is

$$ \rho_n (A) = \frac{ | \{ m \in A : m <  n \} | }{n}   $$
 The \emph{(asymptotic) density} $\rho(A)$ of $A$ is $\lim_{n \to \infty} \rho_n (A)$ if this limit exists.

 While the limit for density does not exist in general, the \emph{upper density}
$$ \overline{\rho}(A)  = \limsup_n \{ \rho_n (A) \} $$

 and the \emph{lower density}

$$ \underline{\rho}(A) = \liminf_n \{ \rho_n (A)\} $$
 always exist.

    We  use $\varphi_e$ for the unary partial function computed by the $e$-th Turing
machine.  Let $W_e$ be the domain of $\varphi_e$.   We identify
a set $A \subseteq \omega$ with its characteristic function $\chi_A$.

  First observe that  \emph{every} Turing degree contains a generically computable set.
Let $A \subseteq \mathbb{N}$.   Let $C(A) = \{ 2^n : n \in A \}$.
Then $C(A)$ is generically computable since the set of powers of $2$
is computable and has density $0$.  All the information about $A$ is in a set of density $0$.
When given $m$, the generic algorithm checks if $m$ is a power of $2$.
If not, the algorithm answers $m \notin C(A)$ and otherwise does not answer.
This example shows that one partial algorithm can generically compute uncountably
many different sets.

  The following sets $R_{k}$ are extremely useful.

\begin{defn}(\cite{JS}, Definition 2.5)
\[ R_k = \{ m :  \ 2^k | m, \ 2^{(k + 1)} \nmid m \}.  \]
\end{defn}

For example, $R_0$ is the set of odd nonnegative integers.  Note that
$\rho(R_k) = 2^{-(k+1)}$.  The collection of sets $\{R_k \}$ forms a
partition of  $\omega - \{0\}$ since these sets are pairwise
disjoint and $ \bigcup_{k = 0}^{\infty} R_k = \omega - \{0\}$.

 From the definition of asymptotic density it is clear that we have
 \emph{finite additivity} for densities.  Of course we do not have 
countable  additivity for densities in general, since
$\omega$ is a countable union of singletons.  However, we do have
countable additivity in the situation where  the 
``tails'' of a sequence contribute vanishingly small density to the
union of a sequence of sets.

\begin{lem}\label{additivity} (\cite{JS}, Lemma 2.6, Restricted countable additivity) \label{rca}
  If $\{ S_i \}, i = 0, 1, \dots $ is a countable collection of pairwise
  disjoint subsets of $\omega$ such that each $\rho(S_i)$ exists
  and $\overline{\rho}( \bigcup_{i = N}^{\infty} S_i ) \to 0$ as $N
  \to \infty$, then

\[   \rho (\bigcup_{i = 0}^{\infty} S_i )  =  \sum_{i=0}^{\infty} \rho(S_i) .  \]

\end{lem}

\begin{defn}(\cite{JS}, Definition 2.7) If $A \subseteq \omega$ then $\mathcal{R}(A) =
  \bigcup_{n \in A} R_n$.
\end{defn}

Our sequence $\{R_n\}$ satisfies the hypotheses of Lemma
\ref{rca}, so we have the following corollary.

\begin{cor}\label{rhoadds}(\cite{JS}, Corollary 2.8) 
$\rho(\mathcal{R}(A)) = \sum_{n \in A} 2^{-(n+1)}$.
\end{cor}

This gives an explicit construction of sets  with pre-assigned densities.
and shows that every real number $r \in [0,1]$ is a density.

\begin{prop} (\cite{JS}, Observation 2.11)  Every nonzero Turing degree contains a set which is not generically computable since the set  $\mathcal{R}(A)$ is generically computable if and only if $A$ is computable.
\end{prop}

\begin{proof} It is clear that  $\mathcal{R}(A)$ is Turing equivalent to $A$. 
If  $\mathcal{R}(A)$ is generically computable by a partial algorithm $\varphi$,
to compute $A(n)$ search for $k \in R_n$ with $\varphi(k) \downarrow$ and output $\varphi(k)$. 
Since $R_n$ has positive density, this procedure must eventually answer, and the answer is correct because $\varphi$ never gives a wrong answer.
\end{proof}

Recall that a set $A$ is \emph{immune} if $A$ is infinite and $A$ does
not have any infinite c.e.~subset and $A$ is \emph{bi-immune} if both
$A$ and its complement $\overline{A}$ are immune. 
It is clear that  no bi-immune set can be  generically computable.

Now the class of bi-immune sets is both comeager and of measure $1$.
This is clear by  countable additivity since the family of sets
containing a given infinite set is of measure $0$ and nowhere dense.
Thus the family of generically computable sets is both meager and of
measure $0$.

\section{Densities and C.E.\ Sets}

  Observe that a set $A$ is generically computable if and only if there
exist c.e.\ sets $B \subseteq A$ and $C \subseteq \overline{A}$ such that
$B \cup C$ has density $1$.    In particular, every c.e.\ set of density $1$
is generically computable.   This suggests the question of how well
c.e.\ sets can be approximated by computable subsets in general.    The following definition
gives two ways to measure how good an approximation is.

\begin{defn}(\cite{DJS}, Definition 3.1)
Let $A, B \subseteq \omega$.
        
        \begin{itemize}
        
                \item [(i)]  Define $d(A,B) = \underline{\rho} (A \bigtriangleup B)$, the lower density of
 the symmetric difference of $A$ and $B$.

               \item[(ii)]  Define $D(A,B) = \overline{\rho} (A \bigtriangleup B)$, the upper density of
 the symmetric difference of $A$ and $B$.
        \end{itemize}
\end{defn}

   To our knowledge the first  result on approximating c.e.\ sets by computable subsets is  a result of Barzdin' \cite{Barzdin} from 1970 showing that for every c. e. set $A$ and every real number $\epsilon > 0$, there is a computable set $B  \subseteq A$ such that  $d(A,B) < \epsilon$.   We thank Evgeny Gordon for bringing this result to our attention. The following result of Downey, Jockusch, and Schupp improves Barzdin's result from $d$ to $D$.

\begin{thm}(\cite{DJS}, Corollary 3.10) \label{approx}
 For every c.e.\ set
  $A$ and real number $\epsilon > 0$, there is a computable set $B
  \subseteq A$ such that  $D(A,B) < \epsilon$.
\end{thm}

  Jockusch and Schupp (\cite{JS}, Theorem 2.22) showed that there is  a c.e.\ set of density $1$ which does  not have any computable subset of density $1$.  It turns out that this  property characterizes an important class of c.e.\ degrees,  where a c.e.\ degree is one which contains a c.e.\ set.  Recall that if ${\bf a}$ is a Turing degree with $A \in {\bf a}$, then
the \emph{jump} of ${\bf a}$, denoted ${\bf a'}$, is the
Turing degree of the halting problem for machines with an oracle for $A$.
 If ${\bf a}$ is a c.e.\ degree then $\mathbf{0' \le \bf a' \le  0''}$.
A  degree  $\bf a$ is \emph{low} if $\mathbf{a'} = \mathbf{0'}$, that is,
${\bf a}'$ is as low as possible.  A degree $\bf a$ is \emph{high} if $\mathbf{a' \geq 0''}$.

Downey, Jockusch, and Schupp \cite{DJS} proved
the following characterization of non-low c.e.\ degrees.  

\begin{thm}(\cite{DJS}, Corollary 4.4)  Let $\mathbf{a}$ be a c.e.\ degree.  Then $ {\bf a}$ is not low if and only if ${\bf a}$ contains a c.e.\ set $A$ of density $1$ with no
computable subset of density $1$.
\end{thm}

  With  Eric Astor they also  proved the following result.

\begin{thm}(\cite{DJS}, Corollary 4.2) There is a c.e.\ set $A$ of density $1$ 
such that the degrees of subsets of $A$ of density $1$ 
are exactly the high degrees.
\end{thm}
  
One of the striking things to emerge from considering density and computability is that there is a
very tight  connection between the positions  of  sets in the arithmetical hierarchy
and the complexity of their densities as real numbers.

Fix a computable bijection between the rationals and $\mathbb{N}$,
so we can classify sets of rationals in the arithmetical hierarchy.

\begin{defn}
  Define a  real number $r$ to be  \emph{left}-$\Sigma^0_n$ if its
  corresponding lower cut in the rationals, $\{q \in \mathbb{Q} : q <
  r\}$, is $\Sigma^0_n$.   Define ``left-$\Pi^0_n$'' analogously.
\end{defn}

 Jockusch and Schupp \cite{JS} proved that a real number $r \in [0 ,1]$ is the
density of a computable set if and only if $r$ is a $\Delta^0_2$ real.
  Downey, Jockusch and Schupp \cite{DJS} carried this much further and  proved
the following results.

\begin{thm}(\cite{JS}, Theorem  2.21  ,  \cite{DJS} Corollary 5.4, Theorems 5.6, 5.7, and 5.13)
  Let $r$ be a real number in the interval $[0,1]$ and suppose that $n \geq 1$. Then the
  following hold:
\begin{itemize}
    \item[(i)]  $r$ is the  density of some  set in $\Delta^0_n$  if and only if   $r$ is left-$\Delta^0_{n+1}$. 
     \item[(ii)] $r$ is the lower density of some  set in $\Delta^0_n$  if and only if   $r$ is left-$\Sigma^0_{n+1}$.
    \item[(iii)] $r$ is the upper density of some  set in  $\Delta^0_n$ if and only if   $r$ is left-$\Pi^0_{n+1}$.
    \item[(iv)] $r$ is the lower density of some  set in $\Sigma^0_n$  if and only if   $r$ is left-$\Sigma^0_{n+2}$.
    \item[(v)] $r$ is the upper  density of some  set in $\Sigma^0_n$  if and only if   $r$ is left-$\Pi^0_{n+1}$.
   \item[(vi)]  $r$ is the      density of some  set in $\Sigma^0_n$  if and only if   $r$ is left-$\Pi^0_{n+1}$. 

\end{itemize}
\end{thm}

  This result  follows by relativization from  characterizing  the densities,  upper
densities, and lower densities of the computable and c.e.\ sets.

\subsection{Asymptotic density and the Ershov Hierarchy}

  The correlation of densities and position in the arithmetical hierarchy is further clarified by considering
densities of sets in the Ershov Hierarchy. The Shoenfield Limit Lemma shows that a set $A$ is $\Delta^0_2$ exactly if there is a computable function $g$ such that for all $x$, $A(x) = lim_s g(x,s)$.   Roughly speaking, the Ershov Hierarchy classifies $\Delta^0_2$ sets by the number of $s$ with $g(x,s) \ne  g(x,s+1)$.
A set $A$ is $n$-c.e.\ if  there exists a computable function $g$ as above such that, for all $x$, 
$g(x,0) = 0$ and there are at most $n$ values of $s$ such that $g(x,s) \neq g(x, s+1)$.

  The $1$-c.e.\ sets are just the c.e.\ sets. 
  The $2$-c.e.\ sets, also called  the d.c.e.\ sets, are  sets   which are the differences of two c.e.\ sets.
Since  the densities of  c.e.\ sets are  precisely the left-$\Pi^0_2$ reals in the unit interval, one is led to suspect that the densities 
of the $2$-c.e.\ sets should  be exactly  the differences of two left-$\Pi^0_2$ reals which are in the unit interval.
  This   is true but there is something to prove since the difference  of $A$ and $B$ may have a density even though $A$ and $B$
do not have densities.  Let  $\mathcal{D}_2$ denote the set of reals which are the difference of two left $\Pi^0_2$ reals.
Downey, Jockusch, McNicholl and Schupp \cite{DJMS}  proved the following results.
  
\begin{thm}(\cite{DJMS}, Corollary 4.3) For every $n \ge 2$, the densities of the $n$-c.e.\ sets coincide with the reals in $\mathcal{D}_2 \cap [0,1]$.
\end{thm}

  It follows that there is a real $r$ which is the density of a $2$-c.e.\ set but not of any c.e.\ or co-c.e.\ set.
  
Say that a $\Delta^0_2$ set $A$ is $f$-c.e.\ if there is a computable function $g$ such that, for all $x$, $g(x,0) = 0$, $A(x) = \lim_s g(x, s)$, and $|\{s : g(x,s) \neq g(x,s+1)\}| \leq f(x)$.

\begin{thm}(\cite{DJMS}, Corollary 5.2) Let $f$ be any computable, nondecreasing, unbounded function.   If $A$ is a $\Delta^0_2$ set that has a density, then the density of $A$ is the same as the density of a set $B$ such that  $B$ is $f$-c.e.
\end{thm}

\subsection{Bi-immunity and Absolute Undecidability}

If $A$ is bi-immune then any c.e.\ set contained in either $A$ or
$\overline{A}$ is finite so being bi-immune is an extreme
non-computability condition.  Jockusch \cite{J69a} proved that
there are nonzero Turing
degrees which do not contain any bi-immune sets.  This raises the
natural question of how strong a non-computability condition can be
pushed into every non-zero degree.  Miasnikov and Rybalov \cite{MR}
defined a set $A$ to be \emph{absolutely undecidable} if every partial
computable function which agrees with $A$ on its domain has a domain
of density $0$.  We might suggest the term \emph{densely undecidable}
as a synonym for ``absolutely undecidable'', since being absolutely
undecidable is a weaker condition than being bi-immune.   The following
beautiful and surprising result is due to Bienvenu, Day and H\"olzl
\cite{BDH}.

\begin{thm} (\cite{BDH})
Every nonzero Turing degree contains an absolutely  undecidable set.
\end{thm}

  The theorem was proved using the Hadamard error-correcting code, which
the authors of \cite{BDH} rediscovered to prove the result.

\section{Coarse Computability}

    The following definitions suggest another quite reasonable concept of ``imperfect computability''. 

\begin{defn}(\cite{JS}, Definition 2.12) Two sets $A$ and $B$ are \emph{coarsely  similar},
  which we denote by  $A \thicksim_c B$, if their symmetric difference $ A
  \bigtriangleup  B = (A \diagdown B ) \cup ( B \diagdown A)$ has density
  $0$. If $B$ is any set coarsely similar to $A$ then $B$ is called a 
\emph{coarse description} of $A$.
\end{defn}

It is easy to check that $\thicksim_c$ is an equivalence relation.
Any set of density $1$ is coarsely  similar to $\omega$, and any
set of density $0$ is coarsely  similar to $\emptyset$.

\begin{defn}(\cite{JS}, Definition 2.13) A set $A$ is \emph{coarsely computable} if $A$ is
  coarsely  similar to a computable set. That is,  $A$ has
a computable coarse description.
\end{defn}

  We can think of coarse computability in the
following way: The set $A$ is coarsely computable if there exists a
\emph{total} algorithm $\varphi$ which may make mistakes on membership in
$A$ but the mistakes occur only on a negligible set. A generic algorithm is always correct when it answers
and almost always answers, while a coarse algorithm always answers and is almost always 
correct.  Note that  all sets of density $1$ or of density $0$ are
coarsely computable. 

   Using the Golod-Shafarevich inequality, Miasnikov and Osin \cite{MO} constructed finitely generated, computably presented groups
whose word problems are not generically computable.  Whether or not there exist finitely presented groups whose word problem is
not generically computable is a difficult open question.  The situation for coarse computability is very different.

\begin{obs} (\cite{JS}, Observation 2.14).  The word problem of any finitely generated group $G =
  \langle X:R \rangle$ is coarsely computable.
\end{obs}

\begin{proof} If $G$ is finite then the word problem is computable.
   If $G$ is an infinite
  group, the set of words on $X \cup X^{-1}$ which are not equal
  to the identity in $G$ has density $1$ and hence is coarsely
  computable. (See, for example, \cite{Woess}.)
\end{proof}

  It is easy to check  that the family of coarsely computable sets is
meager and of measure $0$.  In fact, if $A$ is coarsely computable,
then $A$ is neither $1$-generic nor $1$-random.  This is a consequence of
the fact that 
 if $A$ is $1$-random and $C$ is computable, then the
symmetric difference $A \bigtriangleup C$ is also $1$-random, and the
analogous fact also holds for $1$-genericity.  The result now follows
because $1$-random sets have density $1/2$ (\cite{Nies}), 
 and $1$-generic sets have upper density $1$.

\begin{prop}\label{simple}(\cite{JS}, Proposition 2.15) There is a c.e.~set which is coarsely computable but not generically computable.
\end{prop}

\begin{proof}
  Recall that a c.e.~set $A$ is \emph{simple} if $\overline{A}$ is
  immune.  It suffices to construct a simple set $A$ of density $0$,
  since any such set is coarsely computable but not generically
  computable.   This is done by a slight
  modification of Post's simple set construction.  Namely, for each
  $e$, enumerate $W_e$ until, if ever, a number $> e^2$ appears, and
  put the first such number into $A$.  Then $A$ is simple, and $A$ has
  density $0$ because for each $e$, it has at most $e$ elements less
  than $e^2$.
\end{proof}

  The following construction shows that c.e.\ sets may be neither
generically nor coarsely computable.

\begin{thm} (\cite{JS}, Theorem 2.16) There exists a c.e.~set which is not coarsely similar
  to any co-c.e.~set and hence is neither coarsely computable nor
  generically computable.
\end{thm}

\begin{proof} Let $\{ W_e \}$ be a standard enumeration of all
  c.e.~sets. Let
  \[ A = \bigcup_{e \in \omega} (W_e \cap R_e ) \] Clearly, $A$ is
  c.e.\ We first claim that $A$ is not coarsely similar to any
  co-c.e.~set and hence is not coarsely computable. Note that
  \[ R_e \subseteq A \bigtriangleup \overline{W_e} \] since if $n \in R_e$
  and $n \in A$ , then $n \in (A \diagdown \overline{W_e})$, while if
  $n \in R_e$ and $n \notin A$, then $n \in ( \overline{W_e} \diagdown
  A)$.  So, for all $e$, $(A \bigtriangleup \overline{W_e})$ has positive
  lower density, and hence $A$ is not coarsely similar to
  $\overline{W_e}$.  It follows that $A$ is not coarsely computable.
  Of course, this construction is simply a diagonal argument, but
  instead of using a single witness for each requirement, we use a set
  of witnesses of positive density.

  Suppose now for a contradiction that $A$ were generically
  computable.   Let $W_a$, $W_b$ be
  c.e.~sets such that $W_a \subseteq A$, $W_b \subseteq \overline{A}$,
  and $W_a \cup W_b$ has density $1$.  Then $A$ would be generically
  similar to $\overline{W_b}$ since
  $$A \bigtriangleup \overline{W_b} \subseteq \overline{W_a \cup W_b}$$
  and $\overline{W_a \cup W_b}$ has density $0$. This shows that $A$ is
  not generically computable.
\end{proof}

   We introduce the following construction which will be used repeatedly.  

\begin{defn}(\cite{JS}) Let  $\mathcal{I}_0 = \{0\}$ and for $n > 0$ let $I_n$ be the interval $[n!, (n+1)!)$.  For $A \subseteq \omega$,
 let $\mathcal{I}(A) =  \cup_{n \in A} I_n$.  
\end{defn}

  \begin{thm} \label{maj}(\cite{JS}, proof of Theorem 2.20)
  For all $A$, the set $\mathcal{I}(A)$ is coarsely computable if and only if $A$ is computable.
  \end{thm}

\begin{proof} It is clear that  $\mathcal{I}(A)  \equiv_T A$, so it suffices to show
  that if $A$ is not computable then $\mathcal{I}(A)$ is not coarsely computable.
  If  $\mathcal{I}(A)$ is coarsely computable, we can choose a
  computable set $C$ such that $\rho(C \bigtriangleup \mathcal{I}(A)) = 0$.  The idea is
  now that we can show that $A$ is computable by  using ``majority vote''
  to read  off from $C$
  a set $D$ which differs only finitely from $A$.   Specifically, define
$$D = \{n : |I_n \cap C| > (1/2) |I_n|\}.$$
Then $D$ is a computable set and we claim that $A \bigtriangleup D$ is
finite.   To prove the claim,
assume for a contradiction that $A \bigtriangleup D$ is infinite.  If $n
\in A \bigtriangleup D$, then more than half of the elements of $I_n$ are in
$C \bigtriangleup \mathcal{I}(A)$.  It follows that, for $n \in A \bigtriangleup D$,

$$\rho_{(n+1)!}(C \bigtriangleup \mathcal{I}(A)) \geq \frac{1}{2}\frac{|I_n|}{(n+1)!} = 
\frac{1}{2}\frac{(n+1)! - n!}{(n+1)!} = \frac{1}{2}(1 -
\frac{1}{n+1}).$$ 

  As the above inequality holds for infinitely many
$n$, it follows that $\overline{\rho}(C \bigtriangleup \mathcal{I}(A)) \geq 1/2$, in
contradiction to our assumption that $\rho(C \bigtriangleup \mathcal{I}(A)) = 0$.  It
follows that $A \bigtriangleup D$ is finite and hence $A$ is computable.
\end{proof}

 A similar argument shows that if $A$ is not computable then $\mathcal{I}(A)$ is also
not generically computable. We thus have the following result.

\begin{thm} (\cite{JS}, Theorem 2.20) Every nonzero Turing degree contains a set which is neither
coarsely computable nor generically computable.
\end{thm}

  Since $\mathcal{R}(A)$ is generically computable if and only if $A$ is computable,
it seems natural to ask about the coarse computability of $\mathcal{R}(A)$.
 Post's Theorem shows that  the sets Turing reducible to $0'$ are precisely
the sets which are $\Delta^0_2$ in the arithmetical hierarchy.
Using the limit lemma one can prove the following result. 

\begin{thm}(\cite{JS}, Theorem 2.19) \label{coarseR}  
For all $A$, the set $\mathcal{R}(A) = \bigcup_{n \in A} R_n$ is coarsely computable if and only if  $A \leq_T 0'$ .
\end{thm}

In particular, if $A$ is any noncomputable set Turing reducible to $0'$ 
then $\mathcal{R}(A)$ is coarsely computable but not generically computable.

\section{Computability at densities less than 1}
 
  Generic and coarse computability are computabilites at density $1$. Downey, Jockusch and Schupp \cite{DJS} took the
natural step of considering computability at densities less than $1$.

\begin{defn}(\cite{DJS}, Definition 5.9)  If $ r \in [0,1]$,  a  set $A$ is \emph{partially computable at density } $r$ if
there exists a partial computable function $\varphi$ agreeing with $A(n)$ whenever $\varphi(n)\downarrow$ and with the lower density of domain($\varphi$)  greater than or equal to $ r$.
\end{defn}
  
 A natural first question is:  Are there sets which are computable at all densities $r < 1$
but are not generically computable?
  Actually, we have already seen that every nonzero Turing degree contains such sets.    Any set of the form  $\mathcal{R}(A)$ is partially computable at all densities less than $1$, as Asher Kach observed.     Note that for  any $t \ge 0$,
the set $  \bigcup{R_k }$  where $ k \le t$ and  $k \in A$  is a computable set whose symmetric difference with $\mathcal{R}(A)$ is contained
in $\bigcup\{R_k : k > t\}$, and the latter set has density $2^{-t-1}$.  Furthermore, $\mathcal{R}(A)$  is generically computable if and only if $A$ is computable.

  This   ``approachability'' phenomenon holds very generally.
  
\begin{defn}(\cite{DJS}, Definition 6.9) If  $A \subseteq \omega$, the \emph{partial  computability bound} of $A$
is 
\newline $ \alpha(A) := sup\{r : A \mbox{ is computable at density } r \}$ .
\end{defn}

\begin{thm}(\cite{DJS}, Theorem 6.10)  If $r \in [0,1]$, 
then there is a set $A$ of density $r$ with $\alpha(A) = r$. 
\end{thm}

\begin{proof}  Let $.b_0 b_1 ...$ be the binary expansion of $r$.
By Corollary \ref{rhoadds} the set $D = \bigcup_{ b_i = 1} R_i$
has density $r$. We let $A = D \cup S$ where $S$ is a simple set
of density $0$ (Proposition \ref{simple}).  If $s < r$ we can take enough digits of the expansion
of $r$ so that if  $t = .b_1 \dots b_n$  then $s < t < r$. The set $C$
which is the union of the $R_j$ where  $j \le n, b_j = 1$
is a computable subset of $A$ of density $t$ so $A$ is computable at density $t$.
Since we can take $t$ arbitrarily close to $r$, it follows that $\alpha(A) \geq r$.    To show that
$\alpha(A) \leq r$, assume that  $\varphi$ is a computable partial function which agrees
with $A$ on its domain $W$.   We must show that $\underline{\rho}(W) \leq r$.      For $i \in \{0,1\}$, let   $ T_i = \{n : \varphi(n) = i\}$, so $W = T_0 \cup T_1$.   Then $T_0$ is c.e.\ and $T_0 \subseteq \overline{A} \subseteq \overline{S}$, so $T_0$ is finite because $S$ is simple.   Also $T_1 \subseteq A$, so $\underline{\rho}(T_1) \leq \underline{\rho}(A) = r$, so $\underline{\rho}(W) \leq r$, as needed to complete the proof.
\end{proof}

   In analogy with  partial computability at densities less than $1$, Hirschfeldt, Jockusch, McNicholl and Schupp \cite{HJMS} introduced the analogous concepts for coarse computability. 
We define 
$$ A \triangledown C =  \{ n: A(n) = C(n) \}$$ 
and call $A \triangledown C$ the \emph{symmetric agreement} of $A$ and $C$.   Of course,
the symmetric agreement of $A$ and $C$ is the complement of the symmetric difference of $A$ and $C$.
  
\begin{defn}(\cite{HJMS}, Definition 1.5)  A set $A$ is \emph{coarsely computable at density} $r$ if there is a computable
set $C$ such that the lower density of the symmetric agreement of $A$ and $C$ is at least $r$, that is
$$ \underline{\rho}( A \triangledown C) \ge r $$
\end{defn}

\begin{defn}(\cite{HJMS}, Definition 1.6) If  $A \subseteq \mathbb{N}$, the \emph{coarse computability bound} of $A$ is
$$\gamma (A) : = \sup \{ r : A  \mbox{ \ is coarsely computable at density }   r \}$$
\end{defn}

\begin{prop}  (\cite{HJMS}, Lemma 1.7) 
For every set $A$, $\alpha(A) \leq \gamma(A)$.
\end{prop}
This result follows easily from Theorem \ref{approx}.
    
The next result is due to Greg Igusa and shows that this is the \emph{only} restriction on the
values taken simultaneously by $\alpha$ and $\gamma$.

\begin{thm} (Igusa, personal communication) If $r$ and $s$ are real numbers with $0 \leq r
  \leq s \leq 1$, there is a set $A$ such that $\alpha(A) = r$ and
  $\gamma(A) = s$.
\end{thm}

 The coarse computability bound  of every 1-random  set $A$ is $1/2$.  
This is because for every computable set $C$, the set $A \triangledown C$ is also $1$-random  and so has density $1/2$.
  
Recall that we defined the distance function $D(A,B) = \overline{\rho}(A \bigtriangleup B)$.    It is easily seen that $D$ satisfies the triangle inequality and hence is a pseudometric on Cantor space $2^\omega$.   Since $D(A,B) = 0$ exactly when $A$ and $B$ are coarsely similar, $D$ is actually a metric
 on the space $\mathcal{S}$  of coarse equivalence classes.

 Note that $A$ is coarsely computable  at density $1$ if and only if $A$ is coarsely computable. 
 To exhibit  many sets with $\gamma =1$ which are not coarsely computable, 
 again consider sets of the form $\mathcal{R}(A) = \bigcup_{n \in A} R_n$.
Essentially the same argument as before shows  that $\gamma(\mathcal{R}(A)) = 1$ for every $A$.
For each $k$, use the finite list of which $i \le k$ are in $A$, to answer correctly  on 
$\bigcup_{i = 0}^{k} R_i$ and answer ``yes'' on all $R_l$ with $l >k$.  This algorithm is correct 
with density at least $1 - \frac{1}{2^{k+1}}$.     
  
\begin{lem}(\cite{HJMS}).  For $A \subseteq \omega,  \underline{\rho}(A) = 1- \overline{\rho}(\overline{A})$
\end{lem}
  
For each $n,  \rho_n(A) = 1- \rho_n(\overline{A})$, so the lemma follows by taking the least upper bound of both sides. 
 As a corollary we have
$$ \underline{\rho}( A \triangledown C ) = 1 - D(A,C). $$
So, $\gamma(A) = 1$ if and only if $A$ is a limit of computable sets in the pseudo-metric $D$.
In general, $\gamma(A) = r$ means that the distance from $A$ to the family  $\mathcal{C}$ of computable sets is $1-r$.

\begin{thm}(\cite{HJMS}, Theorems 3.1 and 3.4). For every $r \in (0,1]$ there is a set $A$ with $\gamma(A) = r$  such that $A$ is \emph{not} coarsely computable at density $r$, and a set $B$ such that $\gamma(B) = r$ and $B$ \emph{is} coarsely computable at density $r$.
\end{thm}

   We have seen that if $A$ is not $\Delta^0_2$ then $\mathcal{R}(A)$ is Turing equivalent to $A$,   and $\gamma(\mathcal{R}(A)) = 1$, 
but  $\mathcal{R}(A)$ is not coarsely computable.  Also,  every non-zero c.e.\ degree contains a c.e.\ set $A$ which is generically computable but not coarsely computable (\cite{DJS}, Theorem 4.5).  So the question is whether or not \emph{every} nonzero Turing degree contains a set $A$ such that $\gamma(A) = 1$ but $A$ is not coarsely computable.    The following result gives a negative answer.    The proof includes a crucial lemma due to Joe Miller.

\begin{thm} (\cite{HJMS}, Theorem 5.12)  If $A$ is computable from a $\Delta^0_2$ $1$-generic set 
 and $\gamma(A) = 1$, then $A$ is coarsely computable.
\end{thm}

\begin{thm} (\cite{HJMS}, Theorem 2.1)
Every nonzero (c.e.) degree contains a (c.e.) set $B$ such that
$\alpha(B) = 0$ and $\gamma(B) = \frac{1}{2}$.
\end{thm}
\begin{proof}
   Given $A$, let $B = \mathcal{I}(A)$.   The majority vote argument about $\mathcal{I}(A)$ in the proof of Theorem \ref{maj} actually shows that if $A$ is not computable then $\gamma(\mathcal{I}(A)) \le \frac{1}{2}$.   If $E$ is the set of even numbers, then
   $E \triangledown \mathcal{I}(A)$ has density $1/2$, so $\gamma(\mathcal{I}(A)) \geq \frac{1}{2}$.    Also, it is easily seen $\alpha(\mathcal{I}(A)) = 0$ if $A$ is noncomputable.
 \end{proof}

We  observe that large classes of degrees contain sets $A$ with
$\gamma(A) = 0$.

  A set $S \subseteq 2^{ < \omega}$ of finite binary strings is
  \emph{dense} if every string has some extension in $S$.  Stuart
  Kurtz \cite{K} defined a set $A$ to be \emph{weakly $1$-generic} if
  $A$ meets every dense c.e.\ set $S$ of finite binary strings.

\begin{thm}(\cite{HJMS}, proof of Theorem 2.1.)  If $A$ is a weakly $1$-generic set, then  $\gamma(A) = 0$. 
\end{thm}

\begin{proof}  If $f$ is a computable function then, for each $n,j > 0$, define
  \[
S_{n,j} = \left\{ \sigma \in 2^{< \omega} : |\sigma| \ge j \enspace \&
\enspace \rho_{|\sigma|}(\{k < |\sigma| : \sigma(k) = f(k)\}) <
\frac{1}{n} \right\}.
\]
Each set $S_{n,j}$ is computable and dense.   $A$ meets each
$S_{n,j}$ since $A$ is weakly $1$-generic.   Thus $\{ k: f(k) = A(k)\}$ has lower density $0$.
\end{proof}

  Let $D_n$ be the finite set with canonical index $n$, so $n = \sum \{2^i : i \in D\}$.

Recall that a set $A$ is \emph{hyperimmune} if $A$ is infinite and there is no computable function $f$
such that the sets $D_{f(0)}, D_{f(1)},...$ are pairwise disjoint and all intersect $A$, where
$D_n$ is the finite set with canonical index $n$.    A degree $\mathbf{a}$ is called
\emph{hyperimmune} if it contains a hyperimmune set and otherwise \emph{hyperimmune-free}.
 Kurtz \cite{K} proved that the weakly $1$-generic degrees coincide with
  the hyperimmune degrees. We thus have the following corollary.

\begin{cor}  (\cite{HJMS}, Theorem 2.2)  Every hyperimmune degree contains a set $A$ with $\gamma(A) = 0$.
\end{cor}

A degree $\mathbf{a}$ is called \emph{PA} if every infinite computable tree of binary strings has an infinite $\mathbf{a}$-computable path.

\begin{prop} (\cite{ACDJL}, Proposition 1.8) If $\mathbf{a}$ is PA, then $\mathbf{a}$ contains a set $A$ with $\gamma(A) = 0$.
\end{prop}

\begin{proof}   It is straightforward to construct an infinite computable tree $T$ of binary strings such that the paths through $T$ are exactly the sets $X$ which, on every interval $I_n$, disagree with the
partial computable function $\varphi_n$ on all arguments where the latter is defined.   Then an easy argument shows that $\gamma(X) = 0$ for every path $X$ through $T$, and $T$ has an $\mathbf{a}$-computable path since $\mathbf{a}$ is PA.
\end{proof}

It is easily seen that $\alpha(\mathcal{I}(A)) = 0$ whenever $A$ is noncomputable, and hence every nonzero degree contains a set $B$ such that $\alpha(B) = 0$.   In view of the preceding results on hyperimmune and PA degrees it is natural to ask whether \emph{every}
nonzero degree contains a set $B$ such that $\gamma(B) = 0$.

This question is investigated and answered in the negative in Andrews, Cai, Diamondstone, Jockusch and Lempp \cite{ACDJL}, where the following definition was introduced.

\begin{defn} (\cite{ACDJL})  If $\bf{d}$ is a Turing degree, 
$$\Gamma(\mathbf{d}) = \inf \{ \gamma(A) : A \le_T \mathbf{d} \}  $$
\end{defn}

  Recall that the majority vote argument shows that if $A$ is any noncomputable set then $\gamma(\mathcal{I}(A)) \le 1/2$. 
Therefore  if a Turing degree has a $\Gamma$-value greater than $1/2$ then
it is computable and so has $\Gamma$-value $1$.  

We call a function $g$ a \emph{trace} of a function $f$ if $f(n) \in D_{g(n)}$ for every $n$.

  \begin{defn} (Terwijn, Zambella \cite{TZ}) A set $A$ is \emph{computably
      traceable} if there is a computable function $p$ with the property that
   every $A$-computable function $f$ has a computable trace $g$ such that
   $(\forall n) [|D_{g(n)}| \leq p(n)]$.   (Note that $p$ is independent of $f$.)
 \end{defn}  

  \begin{thm}(\cite{ACDJL}, Theorem 1.10)  If $A$ is computably traceable, then $A$ is
    coarsely computable at density $\frac{1}{2}$.
\end{thm}

The proof is a probabilistic argument.    Since the computably traceable sets are closed downwards under Turing reducibility, it follows easily that $\Gamma(\mathbf{a}) = \frac{1}{2}$ for every degree $\mathbf{a > 0}$ which contains a computably traceable set.

\begin{thm}(\cite{ACDJL}, Theorem 1.12)   If $A$ is a $1$-random set of hyperimmune-free Turing degree and $B \leq_T A$, then $B$ is coarsely computable at density $\frac{1}{2}$.
\end{thm}

   In summary, we know the following.

\begin{itemize}

    \item  $\Gamma(\mathbf{0}) = 1$

    \item  If $\mathbf{a > 0}$, then $\Gamma(\mathbf{a}) \leq \frac{1}{2}$.

    \item  If $\mathbf{a}$ is hyperimmune or PA, then $\Gamma(\mathbf{a}) = 0$.

    \item  If $\mathbf{a}$ is computably traceable and nonzero, then $\Gamma(\mathbf{a}) = \frac{1}{2}$.

    \item If $\bf a$ is both $1$-random and hyperimmune-free, then
      $\Gamma(\mathbf{a}) = \frac{1}{2}$.

\end{itemize}

The following question was raised in \cite{ACDJL}.

\begin{qu}
What is the range of $\Gamma$?   Does it equal $\{0, \frac{1}{2}, 1\}$?
\end{qu}  

Benoit Monin \cite{M} has recently announced the remarkable result that $\Gamma(\mathbf{d})$ is equal to $0$, $1/2$ or $1$ for every degree $\mathbf{d}$.    Together with the results just above, this gives a positive answer to the second half of the above question, and thus a natural trichotomy
of the Turing degrees according to their $\Gamma$-values.   In contrast, Matthew Harrison-Trainor \cite{H} has just announced that the range of the analogue for $\Gamma$ for many-one degrees is $[0, 1/2] \cup \{1\}$.

  Benoit Monin and Andr\'e Nies \cite{MN} have also recently extended and unified some of the above
  results on $\Gamma$ using Schnorr randomness.   In particular they showed the existence of degrees $\mathbf{a}$ 
with $\Gamma(\mathbf{a}) = \frac{1}{2}$ which are neither computably traceable nor $1$-random.   They also gave a new proof of Liang Yu's unpublished result that there are degree $\mathbf{a}$ with $\Gamma(\mathbf{a}) = 0$ such that $\mathbf{a}$ is neither hyperimmune nor PA.

\section{Generic and coarse reducibility and their corresponding degrees.}
  
   One might first consider relative generic computability:  That is,  
 what sets are generically computable by
Turing machines with a full oracle for a set $A$?
  Say that a set $B$ is \emph{generically $A$-computable} if there is a generic computation of $B$  using  a \emph{full} oracle for $A$.
It is easy to see that this notion is not transitive because we start with full information but compute only partial information.   For example, let   $A = \emptyset $ and 
let $B = \{2^n: n \in C \}$ where   $C$ is  any set which is not generically computable.   Then $B$ is generically $A$-computable and $C$ is generically $B$-computable, but $C$ is not generically $A$-computable.
  The following is a remarkable and surprising result of Igusa \cite{Igusa1} showing  there are no minimal pairs for this non-transitive notion of relative generic computability.

\begin{thm}(\cite{Igusa1}, Theorem 2.1) For any noncomputable sets $A$ and $B$  there is a set $C$ which is not  generically computable but which is both generically $A$-computable and generically $B$-computable
\end{thm}

  Generic reducibility (denoted $\leq_{g}$) was introduced by Jockusch and Schupp \cite{JS} (Section 4), and we review the definition here.   A \emph{generic description} of a set $A$ is a partial function $\theta$ which agrees with $A$ on its domain and has a domain of density $1$.
Note that $A$ is generically computable if and only if $A$ has a partial computable generic description.   The basic idea is then that $B \leq_{g} A$ if and only if there is an effective procedure which, from any generic description of $A$, computes a generic description of $B$.    Since computing a partial function is tantamount to enumerating its graph, this is made precise using enumeration operators.   These are similar to Turing reductions but use only \emph{positive} oracle information and also output only positive information.   An \emph{enumeration operator} is a c.e.\ set $W$ of pairs $\langle n, D \rangle$ where $n \in \omega$ and $D$ is a finite subset of $\omega$.      (Here we identify finite sets with their canonical indices and pairs with their codes in saying that $W$ is c.e.   The membership of $\langle n, D \rangle$ in $W$ means intuitively that
from the positive information that $D$ is a subset of the oracle,  $W$ computes that $n$ belongs to the output.)    Hence if $W$ is an enumeration operator and $X \subseteq \omega$, define
$$W^X  :=  \{n : (\exists D)[\langle n, D \rangle \in W \ \& \ D \subseteq X]\}$$
Note that from any enumeration of $X$ one may effectively obtain an enumeration of $W^X$.  If $\theta$ is a partial function, let $\gamma(\theta) = \{\langle a, b \rangle : \theta(a) = b\}$, so $\gamma(\theta)$ is a set of natural numbers coding the graph of $\theta$.   We can now state our formal definition of  generic reducibility.   

\begin{defn} The set $B$ is \emph{generically reducible} to the set $A$ (written $B \leq_{g} A$) if there is a fixed enumeration operator $W$ such that, for every generic description
$\theta$ of $A$,   $W^{\gamma(\theta))} = \theta(\delta)$ for some generic description $\delta$ of $B$.
\end{defn}

This reducibility is also called ``uniform generic reducibility'' and denoted $\leq_{ug}$.   (There is  also a nonuniform version, $\leq_{ng}$, of generic reducibility which we do not consider  in this survey.)

It is easily seen that  $\leq_{g}$ is transitive since the maps induced by enumeration operators are closed under composition.      

\begin{defn}  The \emph{generic degree} of $A$ is $\{ B : B \leq_{g} A \ \& \ A \leq_{g} B\}$.
\end{defn}

  We have seen that the map $\widehat{\mathcal{R}}$ which sends the Turing degree of $A$ to the generic degree of
  $\mathcal{R}(A)$  embeds the Turing degrees into the generic degrees, since any generic algorithm for $\mathcal{R}(A)$ will compute $A$, and the proof of this is uniform.   The generic degrees have a least degree under the ordering induced by $\leq_g$, and this least degree consists of the generically computable sets.

Define $B$ to be \emph{enumeration reducible} to $A$ (written $B \leq_e A$) if there is an enumeration operator $W$ such that $W^A = B$.   

Enumeration reducibility leads analogously to the \emph{enumeration degrees}, i.e. equivalence
classes under the equivalence relation $A \leq_e B$ and $B \leq_e A$.   The Turing degrees can be embedded in the enumeration degrees by the map  which takes the Turing degree of $A$ to the enumeration degree of $A \oplus \overline{A}$.   
   An enumeration degree $\bf a$ is called \emph{quasi-minimal} if it is nonzero and no nonzero enumeration degree $\mathbf{b \leq a}$ is in the range of this embedding.    The following definition is analogous: 

\begin{defn}(\cite{Igusa2} A  generic degree $\boldsymbol{a}$ is \emph{quasi-minimal}
if it is nonzero and no nonzero generic degree $\mathbf{b \leq a}$  is in the range of the embedding $\widehat{\mathcal{R}}$ of the Turing
degrees into the generic degrees defined above.
\end{defn}

The following result gives a connection between quasi-minimality for enumeration degrees and
generic degrees. 

\begin{lem}(\cite{JS}, Lemma 4.9) If $A$ is a set of density $1$ which is not generically computable and the enumeration degree of $A$ is quasi-minimal, then the generic degree of $A$ is also quasi-minimal.
\end{lem}

It is shown in the proof of Theorem 4.8 of \cite{JS} that there is a set $A$ which meets the hypotheses of the lemma.    It follows that there exist quasi-minimal generic degrees which contain sets of density $1$.

 It is therefore natural to consider generic degrees which are \emph{density}-$1$,
that is, generic degrees which contain a set of density $1$ (\cite{Igusa2}).

   A \emph{hyperarithmetical} set is a set computable from any set that can be obtained by iterating the jump operator through the computable ordinals. The class of such sets coincides with the class of $\Delta^1_1$ sets,
which are those sets which can be defined by a prenex formula of second-order arithmetic with all set quantifiers universal and also by a prenex formula with all set quantifiers existential.  Igusa \cite{Igusa2} proves the following striking characterization.

\begin{thm} (\cite{Igusa2}, Theorem 2.15) A set $A$ is hyperarithmetical if and only if there is a density-$1$ set $B$ such that $\mathcal{R}(A) \leq_{g} B$.
\end{thm}

  Cholak and Igusa \cite{CI} consider the question of whether or not every non-zero generic degree bounds a non-zero density-$1$ generic degree.
By the results of \cite{Igusa2} a positive answer would show that there are no minimal generic degrees and a negative
answer would show that there are minimal pairs in the generic degrees.   However, it is not yet known  whether or not
there are minimal degrees or minimal pairs in the generic degrees.

Recall that a  \emph{coarse description} of a set $A$ is a set $C$ which agrees with $A$ on a set of density $1$.
 Hirschfeldt, Jockusch, Kuyper and Schupp \cite{HJKS} introduced  both uniform and nonuniform versions of coarse reducibility
and their corresponding degrees.
  
\begin{defn}(\cite{HJKS}, Definition 2.1) A set $A$ is \emph{uniformly coarsely reducible} to a set $B$, written $A \le_{uc} B$,
if there is a fixed oracle Turing machine $M$ which, given any coarse description of $B$ as an oracle, computes a coarse description of $A$.   A set $A$ is \emph{nonuniformly  coarsely reducible} to a set $B$, written $A \le_{nc} B$ if every  coarse description of $B$ computes a coarse description of $A$.
\end{defn}
 
  These coarse reducibilities induce respective equivalence relations $\equiv_{uc}$ and $\equiv_{nc}$.

\begin{defn}(\cite{HJKS})   The \emph{uniform coarse degree} of $A$ is $\{ B : B \equiv_{uc} A \}$ and
the \emph{nonuniform coarse degree} of $A$ is $\{ B : B \equiv_{nc} A \}$.
\end{defn}
  
   We can embed the Turing degrees into both the nonuniform and the uniform coarse degrees.
   We have already seen that the function $\mathcal{I}$ induces an embedding of the Turing degrees into 
the nonuniform coarse degrees since 
$\mathcal{I}(A) \le_T A$ and each coarse description of $\mathcal{I}(A)$ computes $A$, but the adjustments
needed to compute $A$ depend on the coarse description used.

To construct an embedding of the Turing degrees into the uniform coarse degrees we need more redundancy.  The following map is slightly 
different from but equivalent to the map used in \cite{HJKS}, Proposition 2.3.

\begin{prop}\label{Euniform}(\cite{HJKS}).
Define $\mathcal{E}(A)=\mathcal{I}(\mathcal{R}(A))$.
The function $\mathcal{E}$ induces an embedding of the Turing degrees into the uniform coarse degrees.
\end{prop}

 Recall that a set $X$ is
\emph{autoreducible} if there exists a Turing functional $\Phi$ such
that for every $n \in \omega$ we have $\Phi^{X \setminus \{n\}}(n) =
X(n)$. Equivalently, we could require that $\Phi$ not ask whether its
input belongs to its oracle.   Figueira, Miller and Nies \cite{FMN} showed that
no 1-random set is autoreducible and it is not difficult to  show that no
$1$-generic set is autoreducible.

 Dzhafarov and Igusa \cite{DI} study various notions of ``robust information coding''
and introduced uniform ``mod-finite'', ``co-finite'' and ``use-bounded from below''  reducibilities.
Using the relationships between these  reducibilities and  generic and coarse reducibility, Igusa proved  the following result. 

\begin{thm}(see \cite{HJKS}, Theorem 2.7) If $\mathcal{E}(X) \le_{uc} \mathcal{I}(X)$ then $X$ is autoreducible.
Therefore if  $A$ is 1-random or $1$-generic then
$\mathcal{E}(X) \leq_{nc} \mathcal{I}(X)$ but $\mathcal{E}(X) \nleq_{uc} \mathcal{I}(X)$.
\end{thm}

  There are striking connections between coarse degrees and algorithmic randomness.
The paper \cite{HJKS} shows the following.

\begin{thm}(\cite{HJKS}, Corollary 3.3) If $X$ is weakly $2$-random then $\mathcal{E}(A) \nleq_{nc} X$
for every noncomputable set $A$, so the degree of $X$ is quasi-minimal (in the obvious sense) in both   the uniform and nonuniform coarse degrees.
\end{thm}

For the uniform coarse degrees, this result was strengthened by independently motivated work by Cholak and Igusa \cite{CI} .

\begin{thm} (\cite{CI})  If $A$ is either $1$-random or $1$-generic, then the degree of $A$ is
quasiminimal in the uniform coarse degrees.
\end{thm}

\begin{thm}(\cite{HJKS}, Corollary 5.3) If $Y$ is not coarsely computable and $X$ is weakly $3$-random relative to $Y$, then their  nonuniform coarse degrees form a minimal pair for both uniform and nonuniform coarse reducibility. 
\end{thm}

  Astor, Hirschfeldt and Jockusch \cite{AHJ} introduced ``dense computability'' as a weakening of
both generic and coarse computability.

\begin{defn}(\cite{AHJ}) A set $A$ is \emph{densely computable} (or \emph{weakly partially computable})  if there is a partial computable function $\varphi$
such that $\underline{\rho}(\{n: \varphi(n) = A(n) \}) = 1$.
\end{defn}

 In other words, the partial computable function may diverge on some arguments and give wrong answers on others but agrees with the characteristic function of $A$ on a set of density $1$.   It is obvious that every generically computable set and every coarsely computable set is densely computable.   Note that if $A$ is generically computable but not coarsely computable and $B$ is coarsely computable but not generically computable then $A \oplus B$ is neither generically computable nor coarsely computable,
where, as usual, $A \oplus B = \{ 2n : n \in A \} \cup \{ 2n+1: n \in B \}$.  But $A \oplus B$ is densely computable by using the generic algorithm on even numbers and the coarse algorithm on odd numbers.   Thus dense computability is strictly weaker than the disjunction of coarse computability and generic computability.

We can consider weak partial computability at densities less than $1$.

\begin{defn}(\cite{AHJ}) Let $r \in [0,1]$.   A set $A$ is \emph{weakly partially computable} at density $r$ if there exists a partial computable function 
such that $\underline{\rho}(\{n: \varphi(n) = A(n) \}) \ge r$.
Let 
$$\delta(A) = sup \{ r: A  \mbox{ is weakly partially computable at density } r \}.$$
\end{defn}

It is easy to show the following.

\begin{lem}(\cite{AHJ}) For all $A, \delta(A) = \gamma(A)$.
\end{lem} 

\begin{proof} If $A$ is weakly partially computable at density $r$ by a partial computable function 
$\varphi$,
then by Theorem
\ref{approx} $\text{dom}(\varphi)$ has a computable subset $C$ such that $\underline{\rho}(C) > \underline{\rho}(\text{dom}(\varphi)) - \epsilon$. Let $h$ be the total computable function defined by 
$ h(n)  = \varphi(n) $ if   $n \in C $ and $h(n) = 0$ otherwise.
Since $A \cap C \subseteq \{n:  A(n) = \varphi(n) \} $ it follows that $A$ is coarsely computable at density $r - \epsilon$.
So $\gamma(A) \ge \delta(A)$. Since  $\delta(A) \ge \gamma(A)$ by definition, the two are equal.
\end{proof}

  \begin{defn} A partial function $\Theta$ is a \emph{dense  description of} $A$  if 
 $\{n  : \Theta(n) = A(n)\} $  has density $1$.
\end{defn}

  Using dense descriptions one  can  define  dense reducibility and dense degrees as in \cite{AHJ}.

\end{document}